\documentclass[12pt]{amsart}
\usepackage{amssymb}
\usepackage{hyperref}
\usepackage{mathtools}
\allowdisplaybreaks

\let\iso\cong
\def\cong{\equiv}
\def\det{\mathop{\rm det}\nolimits}
\def\Hom{\mathop{\rm Hom}\nolimits}
\def\odd{\scriptstyle{\rm odd}}
\def\even{\scriptstyle{\rm even}}

\def\e{\varepsilon}
\def\xbar{\bar{x}}
\def\D{\Delta}
\def\tensor{\otimes}
\def\Z{\mathbb{Z}}
\def\Q{\mathbb{Q}}

\def\R{\mathbb{R}}

\def\RorG{{R\,{\rm or}\,G}}
\def\sset{\subseteq}
\def\qf#1{\langle#1\rangle}

\def\legendre#1#2{\bigl(\frac{#1}{#2}\bigr)}
\def\leftcompartment#1{[#1}
\def\rightcompartment#1#2{#1]\llap{\phantom{1}}^{\phantom{1}}_{#2}}
\def\compartment#1#2{[#1]\llap{\phantom{1}}^{\phantom{1}}_{#2}}
\def\I{{\rm I}}
\def\II{{\rm I\!I}}
\newtheorem{theorem}{Theorem}[section]
\newtheorem{lemma}[theorem]{Lemma}
\newtheorem{corollary}[theorem]{Corollary}
\theoremstyle{remark}
\newtheorem*{remark}{Remark}

\newtheorem{example}[theorem]{Example}
\numberwithin{equation}{section}

\newcommand{\defn}[1]{\marginlabel{\footnotesize #1}{\it#1}}
\newcommand{\nodefn}[1]{{\it #1}}
\newcommand{\notation}[1]{\marginlabel{#1}#1}

\renewcommand{\defn}[1]{\nodefn{#1}}
\renewcommand{\notation}[1]{#1}

\begin{document}
\title[The Conway-Sloane 2-adic calculus]{The Conway-Sloane calculus for 2-adic lattices}
\author{Daniel Allcock}
\address{Department of Mathematics\\University of Texas at
  Austin}
\email{allcock@math.utexas.edu}
\urladdr{http://www.math.utexas.edu/\textasciitilde allcock}
\author{Itamar Gal}
\address{Department of Mathematics\\University of Texas at Austin}
\email{igal@math.utexas.edu}
%
\author{Alice Mark}
\address{School of Mathematical and Statistical Sciences\\
  Arizona State University}
\email{amark3@asu.edu}
%
\thanks{First author supported by NSF grant DMS-1101566}
\date{May 4, 2020}

\begin{abstract}
  We motivate and explain the system introduced by Conway and Sloane for
  working with quadratic forms over the $2$-adic integers, and prove
  its validity.  Their system is far better for actual calculations
  than earlier methods, and has been used for many years, but  no proof has been published before now.
\end{abstract}
\maketitle

\section{Introduction}
\label{sec-intro}

\noindent
Our goal in this paper is to explain the system that
Conway and Sloane 
developed for working with lattices (quadratic forms) over 
the ring of $2$-adic integers $\Z_2$.   Algorithms were already known for 
determining when two lattices were isometric, and for finding a canonical form for each one.
But these were clumsy.  In his influential book on quadratic forms, Cassels even wrote 
about $2$-adic integral canonical forms:
``only the masochist is invited to read the rest of this section'' \cite[\S8.4]{Cassels}.
To this day, $2$-adic lattices retain their reputation for complexity.

But
the $2$-adic part of a lattice over~$\Z$ is its most
important part.  Many questions about $\Z$-lattices reduce to
$p$-adic versions of the same questions, where $p$ varies over the primes.  
For example, consider the question of whether one $\Z$-lattice is isometric
to another.  We restrict to the case of rank${ }\geq3$ and some fixed indefinite signature, 
because then it is (almost) true
that an isometry exists if and only if one exists $p$-adically for each $p$.
Most questions
about $p$-adic lattices are easy for odd $p$, including this isomorphism
problem.  So  all the real work takes place at $p=2$.
Other examples of questions with this same flavor are whether a lattice 
represents a given number, or whether
one lattice admits another as a direct 
summand (or as a primitive sublattice).  
See section~\ref{sec-larger-picture} for a little more on this larger picture.

The Conway-Sloane calculus \cite[ch.~15]{SPLAG} is
much simpler than previous approaches to $2$-adic lattices,
for example the original papers on invariants and canonical forms by
Pall \cite{Pall} and Jones \cite{Jones}.  
It is widely used in modern applications,
for example \cite{Allcock}\cite{BEF}\cite{Hohn}\cite{Turkalj}.  
Their innovation was to  introduce  
``oddity fusion'' and ``sign walking'' operations, which are notationally simple and
generate all equivalences.  
Strangely, their formal statement of results (their Theorem~10)
completely avoids these
 operations.   So it has the same
 unwieldy feel as the papers of Pall and Jones just mentioned.
 Proofs of their theorem appear in \cite{Xu}
 and in
Bartels' unpublished
dissertation \cite{Bartels}.
But the literature contains no treatment of the calculus as it is actually used.
We hope
to make it more accessible.
What is new here are the ``givers'' and ``receivers'' of section~\ref{sec-fine-symbols},
and the ``signways'' of section~\ref{sec-2-adic-symbols}.
In particular, we use signways to correct an error in their formulation of canonical forms.

Here is a fairly detailed overview of the calculus.  Our goal is 
to show what it looks like and  what it involves, 
rather than to explain it
properly.  For that, see the formal development beginning in section~\ref{sec-preliminaries}.

\smallskip
\emph{Unimodular lattices:}
The first step in all approaches to $\Z_2$-lattices is to classify the unimodular ones.
Conway and Sloane indicate them by symbols like $L=1^{+2}_2$ or $1^{-3}_3$ or $1^{-4}_\II$.
The main number $1$ says that $L$ is unimodular (over~$\Z_2$).  
If $L$ is even, which is to say that all norms are even, then the subscript is~$\II$.
Otherwise, $L$ is diagonalizable and the subscript is the 
\emph{oddity} $o(L)$ of $L$, meaning the sum mod~$8$ of the diagonal terms
in any diagonalization.  As defined in this paper, the oddity 
is not a lattice invariant in general.
But for unimodular lattices it is, because for them it coincides with
the $2$-signature, which \emph{is} an invariant.  See 
section~\ref{sec-preliminaries} for more discussion.
The superscript
is not a signed number, but rather a sign and a separate nonnegative integer.
The 
integer  is $\dim L$.  The 
sign is
$+$ or~$-$ according to whether $\det(L)\cong\pm1$ or~$\pm3$ mod~$8$.
The sign, dimension and oddity
turn out to determine the isometry class of~$L$.
We prove this in Theorem~\ref{thm-unimodular-lattices}.

An example of the notation: 
the lattices with diagonal inner product
matrices $\qf{1,-1,3}$, $\qf{-1,-1,-3}$ and $\qf{3,3,-3}$
are all unimodular, with determinant~$\pm3$ (up to squares), 
hence sign~$-$.  They also have
dimension~$3$ and oddity $3\in\Z/8$.
So they are isometric to each other, and we write
$1^{-3}_3$ to represent their isometry class.  We built these
lattices by starting with the symbol
$1^{-3}_3$ and choosing three
terms inside $\qf{{\ldots}}$ to have  product~$\pm3$ 
and sum~$3$ 
(both mod~$8$).  In this way it is always
easy to construct representative
lattices for any symbol $1^{\cdots}_{\cdots}$.

The symbols behave cleanly under direct sum: signs multiply
and dimensions and subscripts  add.  For subscripts this means addition in $\Z/8$, 
together with the special rule $\II+t=t$.
For example, $1^{+2}_2\oplus1^{-3}_3\oplus1^{-4}_\II\iso1^{+9}_5$.

\smallskip
\emph{Jordan decompositions:}
A general $\Z_2$-lattice can be expressed 
as a direct sum, where the terms are 
got by rescaling unimodular lattices by  distinct
powers of~$2$.
This is called a Jordan decomposition and the terms are called Jordan constituents.  
Conway and Sloane use symbols like
$1^{+2}_\II$, $2^{-2}_{2}$, $4^{+3}_1$ and $64^{-2}_\II$ to indicate them.  
These lattices
are got from the unimodular lattices with the same decorations, 
namely
$1^{+2}_\II$, $1^{-2}_{2}$, $1^{+3}_1$ and $1^{-2}_\II$,
by scaling inner products
by $1$, $2$, $4$ and $64$ respectively.
The \emph{scale} of each term means this scaling factor.
The \emph{type} is $\I$ or~$\II$ according to whether the 
unimodular lattice is odd or even.
A general $\Z_2$-lattice is a direct sum of such terms, for example
\begin{equation}
  \label{EqExample}
1^{2}_\II\, 
2^{-2}_4 4^{3}_{-1}
16^{1}_1\,
32^{2}_\II\, 
64^{-2}_\II\,
128^{1}_{-1} 256^{1}_1
512^{-4}_\II
\end{equation}
where we have suppressed $+$ signs in superscripts and $\oplus$ symbols between the terms.
A \defn{Jordan symbol} means an expression like \eqref{EqExample},
describing a Jordan
decomposition.
We will use this example many times: it is complicated enough to illustrate
many phenomena.  

\bigskip
There are two main ways that the case of $p$ an odd prime is simpler than the $p=2$ case.
The first is that the unimodular classification is simpler: one needs no subscripts.
The second is that the 
Jordan decomposition is unique up to isometry.  So when $p$ is odd, understanding a $p$-adic lattice
amounts to a writing down something like \eqref{EqExample} without subscripts.  
Equivalences between distinct Jordan decompositions is the subtle part of
$2$-adic lattice theory.  Conway and Sloane introduced \emph{oddity fusion} and
\emph{sign walking} to organize these equivalences.

\smallskip
\emph{Oddity fusion:}
An example of nonuniqueness of Jordan decomposition is
\begin{equation}
  \label{EqOddityFusion}
2^{-2}_4 4^{3}_{-1}
\,\iso\,
2^{-2}_2 4^{3}_{1}
\,\iso\,
2^{-2}_{-2} 4^{3}_{5}
\end{equation}
These are the same except for the oddities of the terms,
and in all three cases the sum of
the oddities is $3$ mod~$8$.  
This illustrates a general phenomenon called oddity fusion:
when the scales of a sequence of Jordan constituents are consecutive powers of~$2$, and the
subscripts are oddities 
rather than ``$\II$'', then those constituents ``share'' their oddities.

To express this more formally
we say two terms of type~$\I$ are in the same \defn{compartment}
if the terms at all intermediate scales also have
type~$\I$.  In example
\eqref{EqExample} there are three compartments: one consisting of the
terms of scales $2$ and~$4$,
one consisting of the term of scale~$16$, and one consisting of
the terms of scales $128$ and~$256$.  
(The scale~$8$ term is unwritten because it is
$0$-dimensional. But it has type~$\II$, hence separates
$2$ and~$4$ from~$16$.)
Usually one indicates the
compartments with brackets, for example
\begin{equation}
  \label{EqBracketedExample}
1^{2}_\II\, 
    [2^{-2}_4 4^{3}_{-1}]
    [16^{1}_1]\,
32^{2}_\II\, 
64^{-2}_\II\,
    [128^{1}_{-1} 256^{1}_1]
512^{-4}_\II
\end{equation}
The brackets are usually omitted for
a compartment consisting of a single term, so here
would omit the brackets that enclose $16^1_1$.

``Oddity fusion'' means that two Jordan symbols $J,J'$, which are
the same except for the subscripts in a compartment, represent
isometric lattices if the sum of the oddities over that compartment
in~$J$ is equal to the corresponding
sum in~$J'$.
Therefore we record this sum (the compartment's oddity)
rather than the oddities of the individual terms.  We attach it
as a subscript to the closing bracket.  
For example,
we write $\compartment{2^{-2} 4^3}{3}$ rather than any 
of the three Jordan symbols in \eqref{EqOddityFusion}.  
This notation
displays less information, while still capturing the
isometry class, so it 
is more canonical.  
After oddity fusion, our example \eqref{EqBracketedExample} becomes
\begin{equation}
  \label{EqFusedExample}
1^2_\II\, 
\compartment{2^{-2} 4^3}{3}
16^1_1\,
32^2_\II\, 
64^{-2}_\II\,
\compartment{128^{1}\, 256^{1}}{0}
512^{-4}_\II
\end{equation}
Most of the simplicity of the
Conway-Sloane approach comes from the use of oddity fusion.
We call a symbol like \eqref{EqFusedExample} a
\emph{$2$-adic symbol}.

\smallskip
\emph{Sign walking:}
Oddity fusion does not generate all equivalences between $2$-adic 
Jordan decompositions.
For example, 
 \eqref{EqFusedExample}
turns out to be isometric to each of
\begin{gather}
  \label{EqWalked1}
\underbracket[.7pt]{1^{-2}_\II\, 
\leftcompartment{2^2}}\rightcompartment{4^3}{-1}
16^1_1\,
32^2_\II\, 
64^{-2}_\II\,
\compartment{128^{1}\, 256^{1}}{0}
512^{-4}_\II
\\
  \label{EqWalked2}
1^2_\II\, 
\compartment{\underbracket[.7pt]{2^{2} 4^{-3}_{\phantom{1}}}}{-1}
16^1_1\,
32^2_\II\, 
64^{-2}_\II\,
\compartment{128^{1}\, 256^{1}}{0}
512^{-4}_\II
\\
  \label{EqWalked3}
1^2_\II\, 
\leftcompartment{2^{-2}}\underbracket[.7pt]{\rightcompartment{4^{-3}}{-1}
16^{-1}_{-3}}\,
32^2_\II\, 
64^{-2}_\II\,
\compartment{128^{1}\, 256^{1}}{0}
512^{-4}_\II
\end{gather}
In each case we have negated
the signs of two adjacent scales
of \eqref{EqFusedExample}, and changed by~$4$ 
the oddity of each compartment involved.  
The underbrackets indicate the terms whose signs were changed.
In \eqref{EqWalked1} and \eqref{EqWalked2}
the only compartment of
\eqref{EqFusedExample}
involved was $[2^{-2}4^3]_3$, so we changed its oddity by~$4$.
In \eqref{EqWalked3} the compartment $16^1_1$ was also involved,
so we also changed its oddity by~$4$.

The rules
for which pairs of terms admit such a \emph{sign walk} are subtle enough that
we postpone them to section~\ref{sec-2-adic-symbols}.  But to illustrate the flexibility they provide,
we show which terms of our example
can  interact with
each other via some chain of sign walks:
\begin{equation}
    \label{EqSignways}
\underbracket[.7pt]{
1^2_\II\, 
\compartment{2^{-2} 4^3}{3}
16^1_1\,
32^2_\II
}
\,\underbracket[.7pt]{
64^{-2}_\II\,
\leftcompartment{128^{1}}}\,
\underbracket[.7pt]{\rightcompartment{256^{1}}{0}
512^{-4}_\II
}
\end{equation}
We call these groups of terms \emph{signways}, suggesting 
highways along which signs can move (or cancel).
In the language of Conway and Sloane, the classification of $2$-adic lattices amounts to the
theorem that  sign walking generates all equivalences between $2$-adic symbols.  
(Theorem~\ref{thm-equivalences-of-2-adic-symbols}.)

Some equivalence relations are like mazes, where
it is not clear which ``moves'' to make when seeking an equivalence between two objects, or
perhaps  only an arcane recipe for these moves is available.
This is the nature of earlier classifications of $2$-adic lattices.  Happily,
sign walking is simple.  For any given $2$-adic symbol, the
sign walks generate an elementary abelian 
$2$-group, acting simply transitively on
the $2$-adic symbols that are equivalent to it.
(See the proof of 
Theorem~\ref{thm-same-signs-imply-isometric-same-as-being-identical}.)
In \eqref{EqSignways} this group is $(\Z/2)^4\times(\Z/2)\times(\Z/2)$.
The $(\Z/2)^4$ factor changes signs in the first signway, 
arbitrarily subject to maintaining the overall sign.
The 
factors $\Z/2$ play the same role for the other signways. 
The alterations of oddities that accompany any given sign
walk are easy to figure out.

One can use sign walking to define a canonical form: 
walk all the $-$ signs as far left as possible,
canceling pairs of such signs when possible.  Then all 
signs will be $+$ except perhaps for
the first terms of the signways.
For \eqref{EqSignways} this canonical form is
\begin{equation*}
1^{-2}_\II\, 
\compartment{2^{2} 4^3}{-1}
16^1_1\,
32^2_\II\, 
64^{-2}_\II\,
\compartment{128^{1}\, 256^{-1}}{4}
512^{4}_\II
\end{equation*}

\medskip
The main virtues of the Conway-Sloane notation are that (i) it allows easy passage
between the notation and the lattices, (ii) it
behaves well under direct sum and scaling, and duality too, (iii) no more
information is displayed than necessary,
and (iv) rather than being constrained to a single canonical form, one
can easily pass between all possible $2$-adic symbols for a particular lattice.
See 
Example~\ref{EgOrthocomplement} for an illustration of (iv): we find all the
$\Z_2$-lattices whose sum with $\qf{2,2}$ is isometric to \eqref{EqFusedExample}.

\medskip
After some (strictly)
motivational background  in section~\ref{sec-larger-picture},
we cover some technical preliminaries in section~\ref{sec-preliminaries}.
Then section~\ref{sec-fine-symbols} defines 
what we call a \emph{fine decomposition}
of a $2$-adic lattice and describes some moves between them.  In section~\ref{sec-Jordan-symbols} 
we classify the
unimodular lattices and introduce oddity fusion.  In section~\ref{sec-2-adic-symbols} we define
$2$-adic symbols and prove that sign walking generates all equivalences between them.
We also discuss canonical forms and how to define some numerical invariants of $2$-adic lattices.
The final section is devoted to the proof of Theorem~\ref{thm-equivalence-of-fine-symbols}.

This note developed from part of a course on quadratic forms given by
the first author at the University of Texas at Austin, with his lecture
treatment greatly improved by the second and third authors.

\section{The larger picture}
\label{sec-larger-picture}

\noindent
This section is meant to describe how the $2$-adic lattice theory fits into 
the larger theory of
integer quadratic forms.  It is not needed later in the paper.

A lattice over $\Z$ or the $p$-adic integers $\Z_p$ means
a free module equipped with a symmetric bilinear pairing that takes values in
the fraction field $\Q$ or $\Q_p$.  
An \emph{isometry} from one such
lattice to another means a module isomorphism that preserves inner products.
In many situations one wants to understand whether two $\Z$-lattices are isometric.  
If $L$ is a $\Z$-lattice, then $L\tensor\Z_p$ is a $\Z_p$-lattice.  If $L'$ is another $\Z$-lattice,
then $L,L'$ are said to lie in the same genus if they have the same signature and
$L\tensor\Z_p$ and $L'\tensor\Z_p$ are isometric for all primes~$p$.  Isometric $\Z$-lattices 
obviously lie in the same genus.

The famous Hasse-Minkowski theory of quadratic forms over~$\Q$
says that two quadratic spaces over~$\Q$
are isomorphic if and only if they are
isomorphic over $\R$ and every~$\Q_p$.  It would be unreasonable
to hope for the corresponding result for lattices over~$\Z$: that
the genus determines the isomorphism class.  What is surprising
is how close to truth this comes.

Until work of Eichler in the 1950s, it was open whether 
the genus determines the isomorphism class of an indefinite $\Z$-lattice
of dim\-en\-sion${ }\geq3$.
Eichler discovered a subtle equivalence relation, whose equivalence
classes are called spinor genera.  Each genus consists of finitely many spinor genera, and each spinor genus
consists of finitely many isometry classes of lattices.  But some mild hypotheses
promote both cases of ``finitely many'' to ``one'':

\begin{theorem}
  \label{ThmDistinctPowers}
    An indefinite genus $G$ of dimension $n\geq3$ 
    consists of exactly one spinor genus,
    unless there exists some prime $p$ such that $G\tensor\Z_p$ is $p$-adically diagonalizable,
    with the $p$-power parts of the diagonal terms all being distinct.
    If $G$ is integral, then 
    this exceptional case can only occur if  $p^{\binom{n}{2}}\mid\det G$.
\end{theorem}

\begin{theorem}[Eichler]
  \label{ThmEichler}
    An indefinite spinor genus of dimension${ }\geq3$ consists of exactly one isometry class.
\end{theorem}

\noindent
Note that the integer $\det G$ and the $\Z_p$-lattice $G\tensor\Z_p$ 
are well-defined, by the definition of a genus.
See \cite[Ch.~15, Thm.~19]{SPLAG}, or the proof of the Corollary to Lemma~3.7 in \cite[Ch.~10]{Cassels},
for Theorem~\ref{ThmDistinctPowers}.
See \cite{Eichler} or \cite[Ch.~10, Thm.~1.4]{Cassels} for Theorem~\ref{ThmEichler}.
The restriction to indefinite forms and dimension${}\geq3$ is essential:
in dimension~$2$ the spinor genus behaves very differently than in
higher dimensions, and for definite forms a genus typically contains
many isomorphism classes.

Except in small dimension, lattices with the distinct-powers-of-$p$
property in Theorem~\ref{ThmDistinctPowers} do not seem to occur in nature.  So
these two theorems form the basis for
our statement in the introduction that for indefinite lattices of dimension${ }\geq3$, it
is ``almost'' true that genera coincide with isometry classes.  
Even if a genus (indefinite of rank${ }\geq3$) 
does have the distinct-powers-of-$p$ property, it might still consist of  a 
single isometry class, and one can check this.  
It is just no longer guaranteed.  

This almost-correspondence between genera and isomorphism classes
is the reason that many questions about 
$\Z$-lattices reduce to $\Z_p$-lattices.  
For $p>2$, a $\Z_p$-lattice
has only one isomorphism class of Jordan decomposition.  And 
each Jordan constituent $J$ 
is characterized
by its scale, dimension and sign.  In this case
there is no subtlety to the isometry classification: to determine
whether two $p$-adic lattices are isometric one just finds Jordan
decompositions and compares them.  
So the $p=2$ case accounts for most of
the isometry analysis.  

(For  odd $p$, the sign of~$J$ is defined as
  the Legendre symbol $\legendre{\det J}{p}=\pm1$, always abbreviated to $\pm$.
Although we did not say so in the introduction, when $p=2$ the sign of $J$ is  
Kronecker's generalization $\legendre{\det J}{2}$ of the Legendre symbol.)  

A second common question about a $\Z$-lattice $L$ is whether a given lattice $M$ occurs a direct summand.
When $L$ is the only lattice in its genus, and the signatures of $M$ and $L$ are compatible,
this reduces to the question of whether  $M\tensor\Z_p$ is a summand of $L\tensor\Z_p$ for all primes~$p$.
For $p>2$ this is easy: $M\tensor\Z_p$ is a summand if and only if 
each constituent of $M\tensor\Z_p$  is lower-dimensional 
than the corresponding constituent of $L\tensor\Z_p$,
or else has the same dimension and sign.  The corresponding 
question for $p=2$ is more subtle---see 
Example~\ref{EgOrthocomplement} for a taste of the required analysis.

A third common question is whether $M$ occurs as a
primitive sublattice of~$L$.  Under the same conditions as in the previous paragraph,
this reduces to the
problem of building a suitable candidate for the orthogonal complement of $M\tensor\Z_p$ in $L\tensor\Z_p$,
for each prime $p$.  The case of odd $p$ is no longer trivial, but still the $p=2$ case usually
dominates the analysis.  See \cite{AllcockE10} for an extended calculation of this sort.

\section{Preliminaries}
\label{sec-preliminaries}

\noindent
Now we begin our formal exposition.
Henceforth, an \defn{integer} means an element of the ring
\notation{$\Z_2$} of $2$-adic integers, and we write \notation{$\Q_2$}
for $\Z_2$'s fraction field.  
We assume known that two odd
elements of $\Z_2$ differ by a square factor 
if and only if they are congruent
mod~$8$.  
Every nonzero $x\in\Q_2$ can be written uniquely as $2^au$ with $a\in\Z$
and $u$
a unit of~$\Z_2$.  We call $u$ the
\defn{odd part} of~$x$.

A \defn{lattice} $L$ means a finite-dimensional
free module over $\Z_2$, equipped with a $\Q_2$-valued symmetric
bilinear form.  
We call $L$ \defn{nondegenerate} if
the natural map $L\to\Hom(L,\Q_2)$ is injective.
In this case, the \defn{dual lattice} $L^*=\Hom(L,\Z_2)$
is naturally 
identified with the set of vectors in $L\tensor\Q_2$
that have integral inner products with all elements of~$L$.
We call~$L$
\defn{integral} if $L\sset L^*$, ie if all 
inner products in~$L$ are integers.  
An integral
lattice is called \defn{even} if all its elements have even norm
(self-inner-product), and \defn{odd} otherwise.  

The \defn{determinant} $\det L$ means the determinant of
the inner product matrix of any basis for~$L$, 
and is well-defined up to multiplication
by squares of units of~$\Z_2$.  In particular, 
the odd part of  $\det L$ is well-defined mod~$8$.
If $L$ is integral then
$\det L=[L:L^*]$ up to a unit of~$\Z_2$.
We call $L$ 
\defn{unimodular} if $L=L^*$; this is equivalent to
$L$ being integral with odd determinant.

The \defn{sign} of a unimodular lattice $U$ means
the Kronecker symbol $\legendre{\det U}{2}$.
Recall that this is defined as $+1$ or $-1$ according to whether $\det U\cong\pm1$ or
$\pm3$ mod~$8$.  We will always abbreviate $\pm1$ to~$\pm$.
The Kronecker symbol has special properties that are important
in quadratic reciprocity.  But these play no role in this paper;
for us it is just a way to record information about  odd numbers
mod~$8$.  We only refer to it
as the Kronecker symbol because it already has that name.

Now consider a lattice got by scaling the inner product on a unimodular
lattice.  We say it has \defn{type} $\I$ or $\II$ according to
whether the unimodular lattice is odd or even. 
For example, $\qf{2}$ has type~$\I$, although it is an even lattice, because
it was got by scaling the odd lattice $\qf{1}$.  On the other hand,
$\bigl(\begin{smallmatrix}4&2\\2&4\end{smallmatrix}\bigr)$ has type~$\II$, because it was got by scaling the
even unimodular lattice
$\bigl(\begin{smallmatrix}2&1\\1&2\end{smallmatrix}\bigr)$.  

The last invariant we need is a $\Z/8$-valued invariant of
quadratic spaces over~$\Q$, called the
\defn{$2$-signature} and written $\sigma_2$.
It is  easy to compute.  
A not-quite-invariant similar to the $2$-signature,
called the \defn{oddity}, is
defined below. It is even easier to compute, and is
what is actually used in the Conway-Sloane calculus.

To compute the $2$-signature of a quadratic space~$V$ over~$\Q_2$,
choose any basis for which the inner product is diagonal.
Then
$\sigma_2(V)\in\Z/8$ is defined as
the sum of the odd parts of 
the diagonal entries, plus~$4$ for each
diagonal entry which is an \defn{antisquare}.
Here an antisquare is defined as a $2$-adic
number of the form $2^{\odd}u$
where $u\cong\pm3$ mod~$8$.
The fact that $\sigma_2(V)$ is independent of the
choice of basis is
surprising.  See \cite[Ch. 15, \S6.1--6.2]{SPLAG} for a proof. 
Some examples:
\begin{align}
    \sigma_2(\qf{1,3,3,7})
    &{}=
    1+3+3+7\cong4\qquad\hbox{(mod~$8$)}
    \notag
    \\
    \sigma_2(\qf{1,3,3,14})
    &{}=
    1+3+3+7\cong4
    \notag
    \\
    \sigma_2(\qf{1,3,6,7})
    &{}=
    1+3+3+7+4\cong0
    \notag
    \\
    \label{EqEvenUnimodularSigZeroA}
\sigma_2\bigl(
    \begin{smallmatrix}
    0&1
    \\
    1&0
    \end{smallmatrix}
    \bigr)
    =
    \sigma_2(\qf{1,-1})
    &{}=
    1-1\cong0
    \\
    \label{EqEvenUnimodularSigZeroB}
    \sigma_2\bigl(
    \begin{smallmatrix}
    2&1
    \\
    1&2
    \end{smallmatrix}
    \bigr)
    =
    \sigma_2(\qf{2,6})
    &{}=
    1+3+4\cong0
\end{align}
In the last two lines we started with even unimodular
lattices, 
diagonalized them
over~$\Q_2$, and then computed~$\sigma_2$.
The $2$-signature is additive: 
$\sigma_2(V\oplus V')=\sigma_2(V)\oplus\sigma_2(V')$ for
any quadratic spaces $V,V'$ over~$\Q_2$.

\medskip
Conway and Sloane \emph{define} the ``oddity'' of~$V$ to be just
another name for~$\sigma_2(V)$.  But  
in actual \emph{use}, 
``oddity'' seems 
to refer to the subscripts used in their
calculus, rather than~$\sigma_2$.  
Our own experience is that this shift of language
is very natural, 
since the subscripts are what one actually uses.
It  seems to be their experience too:
their statement of oddity fusion 
\cite[p.~381]{SPLAG} hints at this,
and this is the only sensible interpretation of
``$\ldots$the total oddity of a
compartment must be changed by $4$ mod~$8$, precisely when$\ldots$''
\cite[p.~382]{SPLAG}.  (For experts:
they are discussing sign walking, and sign walking within
a compartment replaces the corresponding set of Jordan constituents
by new Jordan
constituents, without changing the sublattice they span.
Since the ``oddity'' changes under this, they cannot be referring
to $\sigma_2$ of the sublattice.)

We resolve this conflict by defining oddity
 according to actual use.  We only define it for lattices
 got by scaling unimodular lattices, and for lattices which
 are expressed as direct sums of such lattices.
 So it is a function not on lattices but on lattices
 equipped with such a direct sum decomposition.
 We only speak of the oddity of a lattice when it
 is understood which decomposition we mean.
 If $L$ is got by scaling a unimodular lattice $U$ by
 a power of~$2$, then the
 \defn{oddity} $o(L)$ is defined as $\sigma_2(U)$.
 If $L$ is expressed as a direct sum of rescaled
 unimodular lattices, then $o(L)$ is defined as the
 sum of the oddities of the summands.
 By this definition the oddity is additive:
 $o(L\oplus M)=o(L)+o(M)$.

 The oddity is usually the same as the $2$-signature;
 to describe the difference we begin with 
 direct sum decompositions of a $2$-adic lattice:

\begin{lemma}
    \label{LemGramSchmidt}
    Every $2$-adic lattice is a direct sum of
$1$-dimen\-sion\-al lattices and copies of the lattices
$\bigl(\begin{smallmatrix}0&1\\1&0\end{smallmatrix}\bigr)$ and
  $\bigl(\begin{smallmatrix}2&1\\1&2\end{smallmatrix}\bigr)$, scaled by
    powers of~$2$.
\end{lemma}

\begin{proof}
    After pulling off some summands $\qf{0}$ it is enough to
    treat the nondegenerate case.
    We use induction on dimension and the fact that any sublattice
    with odd determinant is a summand.  
    Suppose a lattice $L$ is given; by scaling we may suppose
    it is integral with some inner product odd.  A vector of
    odd norm  spans a summand,  and then we can appeal to induction.
    So suppose $L$ is even, and choose two elements with odd
    inner product.  
    Their inner product matrix
    $\bigl(\begin{smallmatrix}
    \even&\odd
    \\\odd&\even
    \end{smallmatrix}\bigr)$
    has determinant $(\mathrm{even})(\mathrm{even})-(\mathrm{odd})^2\cong3$ mod~$4$.  
    Therefore they span a summand and we 
    can apply induction to the complementary summand.

    All that remains to prove is: every $2$-dimensional even
    unimodular lattice~$U$ is isomorphic to 
    $
    \bigl(\begin{smallmatrix}
    0&1
    \\
    1&0
    \end{smallmatrix}\bigr)$ 
    or
    $\bigl(\begin{smallmatrix}
    2&1
    \\
    1&2
    \end{smallmatrix}\bigr)$.
    If $U\tensor\Q_2$ has an isotropic vector then 
    so does~$U$.  Choosing
    a primitive one
    as the first basis vector and using row and column operations
    proves
    $
    U\iso
    \bigl(\begin{smallmatrix}
    0&1
    \\
    1&0
    \end{smallmatrix}\bigr)$.
    This argument applies in particular if $\det U$ is
    (in the square class of) $-1$.  This is because an orthogonal
    basis for $U\tensor\Q_2$
    has isotropic inner product matrix
    $\bigl(\begin{smallmatrix}
    t&0
    \\
        0&-t\cdot\mathrm{square}
    \end{smallmatrix}\bigr)
    \iso
    \bigl(\begin{smallmatrix}
    t&0
    \\
        0&-t
    \end{smallmatrix}\bigr)$.

    Now suppose $U$ is anisotropic.  We have seen $\det U\cong3$ mod~$8$.
    Therefore $U$ has an inner product matrix of the form
    $\bigl(\begin{smallmatrix}
    2u&1
    \\
    1&2v
    \end{smallmatrix}\bigr)$
    where $u,v$ are units.  In fact
    $\det U\cong3$ mod~$8$ forces $u\cong v$ mod~$8$.  The vectors
    $(1,1)$, $(2,1)$ and $(3,1)$ have norms${}\cong 6u$, $-2u$ and $-6u$
    mod~$16$.  Therefore $U$ has a vector of norm${}\cong 2$ mod~$16$.
    Rescaling it gives a norm~$2$ lattice
    vector.  By rescaling a supplementary
    basis vector we may suppose the determinant is exactly~$3$.  
    Then a row/column
    operation lets us take the off-diagonal terms to be~$1$, after
    which
    $\det U=3$ forces 
    $U\iso\bigl(\begin{smallmatrix}
    2&1
    \\
    1&2
    \end{smallmatrix}\bigr)$.
\end{proof}

Suppose $L$ is a unimodular lattice decomposed as a direct sum
corresponding to a diagonalization
$\qf{d_1,\dots,d_n}$, with the $d_i$ units in~$\Z_2$.
By definition, both $L$'s oddity and $2$-signature
are $d_1+\cdots+d_n$ (mod~$8$).
Every even unimodular lattice is a sum of
copies of summands
$\bigl(\begin{smallmatrix}
0&1
\\
1&0\end{smallmatrix}
\bigr)$ and
$
\bigl(\begin{smallmatrix}
2&1
\\
1&2
\end{smallmatrix}\bigr)$.
Our calculations 
\eqref{EqEvenUnimodularSigZeroA}--\eqref{EqEvenUnimodularSigZeroB} 
shows that such a lattice has
$2$-signature~$0$, hence also oddity~$0$.

Scaling a unimodular lattice
by a power of~$2$ might change the $2$-signature,
by introducing or eliminating antisquares.
But it leaves the oddity alone.
For example, scaling $\qf{3}$ by~$2$ to get $\qf{6}$ changes
the $2$-signature from $3$ to~$7$ but leaves the oddity
equal to~$3$.
The general rule is:
any direct sum decomposition as in
Lemma~\ref{LemGramSchmidt} has oddity equal to
the sum (mod~$8$) of the odd parts of the 
$1$-dimensional terms.
We will discuss oddity further when we introduce
fine decompositions in section~\ref{sec-fine-symbols} and Jordan
decompositions in section~\ref{sec-Jordan-symbols}.

\medskip
For unimodular lattices, the dimension, sign, type and 
oddity turn out to be a complete set
of invariants.  We prove this in
Theorem~\ref{thm-unimodular-lattices}.  Conway and Sloane express the
isometry class of a unimodular lattice as
$1^{\pm n}_t$ where $\pm$ is the sign, $n$ is the dimension and $t$ is either
the formal symbol $\II$ (for even lattices)
or the oddity (for odd lattices).
In particular, the subscript implicitly records the parity of
the lattice.
We just saw that all even unimodular
lattices have oddity~$0$, so in this case there is no point 
recording it.

If $q$ is a power of~$2$ then 
(after Theorem~\ref{thm-unimodular-lattices})
we will write $q^{\pm n}_\II$ or $q^{\pm
  n}_t$ for the lattice got from $1^{\pm n}_\II$ or $1^{\pm n}_t$ by
rescaling all inner products by~$q$.  For example, $2^{-2}_\II$ has
inner product matrix
$\bigl(\begin{smallmatrix}4&2
\\2&4\end{smallmatrix}\bigr)$.  
The subscript records 
whether the lattice has type~$\I$ or~$\II$, which we recall
is the scale-invariant
generalization of the oddness/evenness of unimodular lattices.
The
number~$q$ is called the \defn{scale} of the symbol (or lattice).
Although one quickly learns the rules,
the following table lets one read off the 
oddity and $2$-signature of any sum of scaled unimodular lattices:
\begin{equation}
    \label{EqTableOfOddities}
    \begin{matrix}
        L
        \lower3pt\hbox{\strut}
        & 1^{\cdots}_t,2^{+\cdots}_t 
        & 2^{-\cdots}_t 
        & 1^{\cdots}_{\II},2^{\cdots}_{\II}
        \\
        o(L)
        & t & t & 0
        &\rlap{additive, and $o(2L)=o(L)$}
        \phantom{\hbox{additive, and $\sigma_2(4L)=\sigma_2(L)$}}
        \\
        \sigma_2(L)
        & t & t+4 & 0
        &\hbox{additive, and $\sigma_2(4L)=\sigma_2(L)$}
    \end{matrix}
\end{equation}

Except for special cases, 
we will not use this notation until we 
have classified the unimodular lattices 
in Theorem~\ref{thm-unimodular-lattices}.
The special cases are in dimension~$1$ and 
the even case in dimension~$2$:
for $q$ any power of~$2$ we define
\begin{equation*}
    \begin{matrix}
        \lower3pt\strut
        &q^{+1}_1&q^{+1}_{-1}&q^{-1}_3&q^{-1}_{-3}&
        q^{+2}_\II&q^{-2}_\II
        \\
        \phantom{\textrm{with oddity}}\llap{\textrm{as}}
        &
        \qf{q}&\qf{-q}&\qf{3q}&\qf{-3q}&
        \bigl(\begin{smallmatrix}0&q
            \\
        q&0\end{smallmatrix}\bigr) 
        &
\bigl(\begin{smallmatrix}2q&q
\\q&2q\end{smallmatrix}\bigr)
    \\
    \raise2pt\strut\textrm{with oddity}
        &
    1&-1&3&-3&0&0
    \end{matrix}
\end{equation*}
These definitions are
compatible with the more general notation.
We will usually omit the
symbol $\oplus$ from direct sums, for example
writing $1^{+1}_{-1}\,1^{-1}_3\,4^{+2}_\II$ for $1^{+1}_{-1}\oplus1^{-1}_{3}\oplus4^{+2}_{\II}$.  
To lighten the notation one
usually suppresses plus signs in superscripts, for example
$1^1_{-1}\,1^{-1}_3\,4^2_\II$, and/or suppresses the dimensions when they are~$1$, for
example $1^+_{-1}\,1^-_3\,4^{2}_\II$.
One could suppress even more, such as leaving the subscript blank for summands of type~$\II$.  But excessive abbreviation is more error-prone than helpful.

\section{Fine symbols}
\label{sec-fine-symbols}

\noindent
In this section we work with a finer decomposition of a lattice than the usual
Jordan decomposition.  The goal is to establish that certain ``moves'' 
between such decompositions do not change the isometry class of the lattice.
This will make the corresponding facts for Jordan decompositions
in the next section easy to state and prove.
Theorem~\ref{thm-equivalence-of-fine-symbols}, proven in section~\ref{sec-equivalences-between-fine-decompositions}, captures the full classification
of $2$-adic lattices, but in a very clumsy way.  The rest
of this paper recasts this classification into a simpler
form.

By a \defn{fine decomposition} of a lattice $L$ we mean a direct sum
decomposition in which each summand (or \defn{term}) is one of
$q^{1}_{\pm1}$, $q^{-1}_{\pm3}$ or $q^{\pm2}_\II$, with the last case
only occurring if every term of that scale has type~$\II$.
The name reflects the fact that no
further decomposition of the summands is possible.
By \eqref{EqTableOfOddities}, 
the oddity of (this decomposition of) $L$ can be read off as
the sum mod~$8$ of the numerical subscripts.
And the $2$-signature of~$L$ can be got from that by adding~$4$
for each term $q^{-\cdots}_{\cdots}$ with $q=2^{\odd}$.
A fine decomposition always exists, by starting
with a decomposition as a sum of $q^{\pm1}_t$'s and $q^{\pm2}_\II$'s
(Lemma~\ref{LemGramSchmidt})
and applying the next lemma repeatedly.

\begin{lemma}
\label{lem-odd-unimodular-admits-orthogonal-basis}
If $\e,\e'$ are signs then 
$1^{\e1}_t\,1^{\e'2}_\II$ admits an orthogonal basis.
\end{lemma}

\begin{proof}
Write $M$ and $N$ for the two summands and consider
the three elements of $(M/2M)\oplus(N/2N)$ that lie in neither $M/2M$
nor $N/2N$.  Any lifts of them have odd norms and even inner
products.  Applying row and column operations to their inner product
matrix leads to a diagonal matrix with odd diagonal entries.
\end{proof}

In order to discuss the relation between distinct fine decompositions
of a given lattice, we introduce the following special language for
$1$-dimensional lattices only.  We call $q^{+1}_1$ and $q^{-1}_{-3}$
``givers'' and $q^{+1}_{-1}$ and $q^{-1}_3$ ``receivers''.  
(Type~$\II$ lattices are neither givers nor receivers.)
The idea
is that a giver can give away two oddity and remain a meaningful symbol
($q^+_1\to q^+_{-1}$ or $q^-_{-3}\to q^-_3$), while a receiver can accept
two oddity.  We often use a subscript $R$ or $G$ in place of the
oddity, so that $1^+_G$ and $1^-_G$ mean $1^+_1$ and $1^-_{-3}$, while
$1^+_R$ and $1^-_R$ mean $1^+_{-1}$ and $1^-_3$.  Scaling inner
products by $-3$
negates signs and preserves giver/receiver status, while scaling them by
$-1$ preserves signs and reverses giver/receiver status.

A \defn{fine symbol} means a sequence of
symbols $q^{\pm2}_\II$ and $q^{\pm}_{\RorG}$.  We replace $R$ and $G$ by numerical subscripts
whenever convenient, and 
regard two symbols as the same if they differ by permuting
terms.  Two scales are called {\it adjacent} if they differ by a factor of~$2$.

\begin{lemma}[Sign walking]
\label{lem-sign-walking-for-fine-symbols}
Consider a fine symbol and two terms of it that satisfy one of the
following conditions:
\begin{enumerate}
\setcounter{enumi}{-1}
\item
\label{item-sign-walking-same-scale}
they have the same scale;
\item
\label{item-sign-walking-adjacent-scales-I-and-II}
they have adjacent scales and different types;
\item
\label{item-sign-walking-adjacent-scales-both-I}
they have adjacent scales and are both givers or both receivers;
\item
\label{item-sign-walking-scales-differ-by-factor-of-4}
their scales differ by a factor of~$4$ and they both have type~$\I$.
\end{enumerate}
Consider as well the fine symbol got by negating the signs of these two
    terms,
and  in case
\eqref{item-sign-walking-adjacent-scales-both-I}
also changing both from givers to receivers or vice-versa.  Then
the two fine symbols represent  isometric lattices.
\end{lemma}

An alternate name for
\eqref{item-sign-walking-scales-differ-by-factor-of-4} might be
sign jumping.  Conway and Sloane informally describe 
it as a
composition of two sign walks of type
\eqref{item-sign-walking-adjacent-scales-I-and-II}. For example,
$$
1^1_12_\II^{+0}4^1_1
\to
1^{-1}_{-3}2_\II^{-0}4^1_1
\to
1^{-1}_{-3}2_\II^{+0}4^{-1}_{-3}.
$$ 
They also observe that this doesn't really make sense: $2_\II^{-0}$ is
illegal because the $0$-dimensional lattice has determinant~$1$, hence sign~$+$.

\begin{proof}
It suffices to prove the following isometries, where $\e,\e'$ are
signs, $X$ represents $R$ or $G$, and $X'$ represents $R$ or $G$:

\eqref{item-sign-walking-same-scale}
$1^{\e2}_\II \, 1^{\e'2}_\II \iso 1^{-\e2}_\II \, 1^{-\e'2}_\II$
and
$1^{\e}_X \, 1^{\e'}_{X'} \iso 1^{-\e}_X \, 1^{-\e'}_{X'}$

\eqref{item-sign-walking-adjacent-scales-I-and-II}
$1^{\e2}_\II \, 2^{\e'}_{X'} \iso 1^{-\e2}_\II \, 2^{-\e'}_{X'}$
and
$1^{\e'}_{X'} \, 2^{\e2}_\II \iso 1^{-\e'}_{X'} \, 2^{-\e2}_\II$

\eqref{item-sign-walking-adjacent-scales-both-I}
$1^{\e}_G \, 2^{\e'}_G \iso 1^{-\e}_R \, 2^{-\e'}_R$

\eqref{item-sign-walking-scales-differ-by-factor-of-4}
$1^{\e}_X \, 4^{\e'}_{X'} \iso 1^{-\e}_X \, 4^{-\e'}_{X'}$

\noindent
The first part of \eqref{item-sign-walking-same-scale} is
trivial except for the assertion $1^{+2}_\II \, 1^{+2}_\II\iso
1^{-2}_\II \, 1^{-2}_\II$.  Choose a
norm~$4$ vector $x$ of the right side.
    Then choose $y$
to have inner product~$1$ with~$x$.  The span of $x$ and $y$ is even
of determinant${}\cong-1$ mod~$8$, so it is a copy of $1^{+2}_\II$.  Its orthogonal
complement must also be even unimodular, hence one of $1^{\pm2}_\II$,
hence $1^{+2}_\II$ by considering the determinant.

The second part of \eqref{item-sign-walking-same-scale} is best
understood using numerical subscripts: we must show
$1^{\e}_t\,1^{\e'}_{t'}\iso1^{-\e}_{t+4}\,1^{-\e'}_{t'+4}$, i.e., $\qf{t,t'}\iso\qf{t+4,t'+4}$.
To see this, note that the left side
represents $t+4t'\cong t+4$ mod~$8$, that this is odd and therefore
corresponds to some direct summand, and the determinants of the two
sides are equal.  Note that givers and receivers always have oddities
congruent to $1$ and $-1$ mod~$4$ respectively, so changing a numerical subscript
by~$4$ doesn't alter giver/receiver status.  
Furthermore, the sign on $1^\e_t$ changes since exactly
one of $t,t+4$ lies in $\{\pm1\}$ and the other in $\{\pm3\}$, and 
similarly for $1^{\e'}_{t'}$.
The same argument works
for 
\eqref{item-sign-walking-scales-differ-by-factor-of-4}, in the form
$1^{\e}_t\,4^{\e'}_{t'}\iso1^{-\e}_{t+4}\,4^{-\e'}_{t'+4}$.

For the first part of
\eqref{item-sign-walking-adjacent-scales-I-and-II} we choose a basis
for $1^{\e2}_\II$ with inner product matrix
$\bigl(\begin{smallmatrix}2&1\\1&0\,{\rm
    or}\,2\end{smallmatrix}\bigr)$ where the lower right corner
  depends on $\e$.  Replacing the second basis vector by its sum with
  a generator of $2^{\e'}_{X'}$ changes the lower right corner by $2$
  mod~$4$.  This toggles the $2\times2$ determinant between $-1$ and
  $3$ mod~$8$.  Therefore it gives an even unimodular summand of
  determinant $-3$ times that of $1^{\e2}_\II$, hence of   sign
  $-\e$.
  Since the overall determinant is an invariant, 
  the determinant of its complement is therefore $-3$ times that of
  $2^{\e'2}_{X'}$.  So the complement is got from $2^{\e'2}_{X'}$ by
  scaling by~$-3$.  We observed above that scaling by $-3$ negates the
  sign and preserves giver/receiver status, so the complement is $2^{-\e'2}_{X'}$.  The
  second part of \eqref{item-sign-walking-adjacent-scales-I-and-II}
  follows from the first by passing to dual lattices and then scaling
  inner products by~$2$. (It is easy to see that the dual lattice has
  the same symbol with each scale replaced by its reciprocal.)

\eqref{item-sign-walking-adjacent-scales-both-I} After
rescaling by~$-3$ if
necessary to take $\e=+$, it suffices to prove
$1^+_G\,2^{\e'}_G\iso1^-_R\,2^{-\e'}_R$, i.e., $\qf{1,2}\iso\qf{3,6}$
and $\qf{1,-6}\iso\qf{3,-2}$. 
In each case one finds a vector on the left side whose norm is odd and appears on the
right, and then compares determinants.
\end{proof}

Further equivalences between fine symbols are phrased in terms of
``compartments''.  A \defn{compartment} means a set of type~$\I$
terms, the set of whose scales forms a sequence of consecutive
powers of~$2$, and which is maximal with these properties.  For
example in $1^2_\II\, 2^-_G\, 2^-_R\, 4^+_G\, 16^-_R$,
the set of scales that have type~$\I$ are $\{2,4,16\}$.  These fall into
two strings of consecutive powers of~$2$, namely $\{2,4\}$ and $\{16\}$.  So
there are two  compartments: $2^-_G\, 2^-_R\, 4^+_G$ and 
$16^-_R$.

\begin{lemma}[Giver permutation and conversion]
\label{lem-giver-permutation-and-conversion}
Consider a fine symbol and the
symbol obtained by one of the following operations.
Then the lattices they represent are isometric
    and have the same oddity.
\begin{enumerate}
\item
\label{item-giver-permutation} Permute
        the subscripts $G$ and $R$ within a
        single
compartment.
\item
\label{item-giver-conversion} Convert any four $G$'s in a
compartment 
to $R$'s,
or vice versa.
\end{enumerate}
\end{lemma}

\begin{proof}
Giver permutation, meaning operation \eqref{item-giver-permutation}, can be achieved by repeated use 
of the 
isomorphisms
$
1^\e_G \, 1^{\e'}_R \iso 1^\e_R \, 1^{\e'}_G
$
and
$
1^\e_G \, 2^{\e'}_R \iso 1^\e_R \, 2^{\e'}_G
$
(scaled up or down as necessary).
To establish these we first rescale by $-3$ if
necessary, to take $\e=+$ without loss of generality.  This leaves the cases
$\qf{1,-1}\iso\qf{-1,1}$,   
$\qf{1,3}\iso\qf{-1,-3}$, 
$\qf{1,-2}\iso\qf{-1,2}$ and $\qf{1,6}\iso\qf{-1,10}$.  One proves
each by finding a vector on the 
left whose norm is odd and appears on the
right, and then comparing determinants.
To see the invariance of the oddity,
imagine
the giver giving $2$ oddity to the receiver.
This converts the giver to a receiver and vice-versa.

For giver conversion, meaning operation \eqref{item-giver-conversion},
the oddity is invariant by the same argument as for
giver permutation.  It remains to prove the isomorphism of the
    lattices.
We assume first that more than one scale
is present in the compartment, so we can
choose terms of adjacent scales.  Assuming four $G$'s are
present in the compartment, we permute a pair of them to our chosen terms, then use sign walking to convert 
these terms  to
receivers.  This negates both signs.  Then we permute these $R$'s away,
replacing them by the second pair of $G$'s, and repeat the sign walking.  This converts
the second pair of $G$'s to $R$'s and restores the original signs.

For the case that only a single scale is present we first treat what
will be the essential cases, namely
\begin{equation*}
  1^+_G\,1^+_G\,1^+_G\,1^+_G\iso1^+_R\,1^+_R\,1^+_R\,1^+_R
  \hbox{\ \ and\ \ }
1^-_G\,1^+_G\,\discretionary{}{}{}1^+_G\,1^+_G\iso1^-_R\,1^+_R\,1^+_R\,1^+_R
\end{equation*}  That is,
\begin{equation*}
  \qf{1,1,1,1}\iso\qf{-1,-1,-1,-1}
  \hbox{\ \ and\ \ }
\qf{-3,1,1,1}\iso\qf{3,-1,-1,-1}
\end{equation*}
In the first case we exhibit a suitable
basis for the left side, namely 
$(2,1,1,1)$ and the images of $(-1,2,1,-1)$ under cyclic permutation
of the last~$3$ coordinates.  In the second we note that the left side
is the orthogonal sum of the span of $(1,0,0,0)$ and $(0,1,1,1)$,
which is a copy of $\qf{-3,3}$, and the span of $(0,-1,1,0)$ and
$(0,0,1,-1)$, which is a
copy of $1^{-2}_\II$.  Since each of these is isometric to
its scaling by~$-1$, so is their direct sum.

Now we treat the general case when only a single scale is present.
Suppose there are at least~$4$ givers.  By scaling by a power of~$2$
it suffices to treat the unimodular case.  By sign walking we may
change the signs on any even number of them, so we may suppose at most
one~$-$ is present.  (Recall that sign walking between terms of the
same scale doesn't affect subscripts $G$ or~$R$.)  By the previous paragraph we
may convert four $G$'s to $R$'s.  Then we reverse the sign walking
operations to restore the original signs.
\end{proof}

The following theorem captures the full classification of
$2$-adic lattices.  It is already simpler than the
results in \cite{Jones} and \cite{Pall}.
But fine symbols package information poorly, and much
greater simplification is possible.  
We will develop this in the next two sections.

\begin{theorem}[Equivalence of fine symbols]
\label{thm-equivalence-of-fine-symbols}
Two fine symbols represent isometric lattices if and only if they are
related by a sequence of sign walking, giver permutation and giver
conversion operations.
\end{theorem}

Although it is natural to state the theorem here, its proof depends on
Theorem~\ref{thm-unimodular-lattices}.  The first place we use it is
to prove Theorem~\ref{thm-equivalences-of-2-adic-symbols}, so
logically the proof could go anywhere in between.  But in fact we 
defer it to
section~\ref{sec-equivalences-between-fine-decompositions} to avoid
breaking the flow of ideas.

\section{Jordan symbols}
\label{sec-Jordan-symbols}

\noindent
In this section we define and study the Jordan decompositions of a lattice.
The main point is that
``oddity fusion'' neatly wraps up all the giver permutation and conversion operations
from the previous section.  
We begin by classifying the unimodular lattices:

\begin{theorem}[Unimodular lattices]
\label{thm-unimodular-lattices}
A unimodular lattice is characterized by its dimension, type, sign
and oddity.
\end{theorem}

Recall that for unimodular lattices, the 
oddity is defined to be the $2$-signature.  Since the
$2$-signature is a genuine invariant of lattices,
the oddity is a genuine invariant of unimodular lattices.
Also, recall from 
\eqref{EqEvenUnimodularSigZeroA}--\eqref{EqEvenUnimodularSigZeroB}
that all even unimodular lattices have vanishing
$2$-signature (hence vanishing oddity).  

\begin{proof}
Consider unimodular lattices $U,U'$ with the same dimension, type,
sign and oddity, and fine symbols $F,F'$ for them.  The product of the
signs in $F$ equals the sign of $U$, and similarly for $U'$.  
Since $U$ and $U'$ have the same sign, we may use sign
walking to make the signs in $F$ the same as in~$F'$.  If $U,U'$ are
even then the terms in $F$ are now the same as in $F'$, so $U\iso U'$.  So
suppose $U,U'$ are odd.

By giver permutation, and exchanging $F$ and $F'$ if necessary, we may
suppose that all non-matching subscripts are $R$ in $F$ and $G$ in
$F'$.  And by giver conversion we may suppose that the number of
non-matching subscripts is $k\leq3$.  Since changing a receiver to a
giver without changing the sign increases the oddity by two,
$o(U')=o(U)+2k$.  Since $o(U')\cong o(U)$ mod~$8$ we have $k=0$.  So
the terms in $F$ are the same as in $F'$, and $U\iso U'$.
\end{proof}

We now have license to use the notation $q^{\pm n}_t$ and $q^{\pm
  n}_\II$ from section~\ref{sec-preliminaries}.  We say that such a symbol is
\defn{legal} if it represents a lattice.  The legal symbols are
\begin{gather*}
  q^{+0}_\II
  \\
  \hbox{$q^{\pm n}_\II$ with $n$ positive and even}
  \\
  \hbox{$q^{+1}_{\pm1}$ and $q^{-1}_{\pm3}$}
  \\
  \hbox{$q^{+2}_0$, $q^{+2}_{\pm2}$, $q^{-2}_4$ and $q^{-2}_{\pm2}$}
  \\
  \hbox{$q^{\pm n}_t$ with $n>2$ and $t\cong n$ mod~$2$}
\end{gather*}
A good way to mentally organize these is to regard the
conditions for dimension${}\neq1,2$ as obvious, remember that $q^2_4$
and $q^{-2}_0$ are illegal, and remember that the subscript of
$q^{\pm1}_t$ determines the sign.

The illegality of $1^2_4$ and $1^{-2}_0$ follows by considering 
all possible sums $1^{\e1}_t1^{\e'1}_{t'}$.  
When the signs $\e,\e'$ 
are the same, either both subscripts are in $\{\pm1\}$ or both are
in $\{\pm3\}$, so the total oddity cannot be~$4$.
When the signs are different, one subscript is $\pm1$ and the other is $\pm3$, so
the total oddity cannot be~$0$.  

This calculation used the simple rules for direct sums of unimodular lattices: signs
multiply and dimensions and subscripts add, subject to the special
rules $\II+\II=\II$ and $\II+t=t$.  

\medskip
A \defn{Jordan decomposition} of a lattice means a direct sum
decomposition whose summands (called \defn{constituents})
are unimodular lattices scaled by distinct powers of~$2$.  By the
\defn{Jordan symbol} for the decomposition we mean the list of the symbols
(or \defn{terms})
$q^{\pm n}_\II$ and $q^{\pm n}_t$ for the summands.  An example we
will use in this section and the next, and mentioned already in the introduction, is 
\begin{equation}
  \label{EqBasicExample}
1^2_\II\, 
2^{-2}_6
4^3_{-3}\, 
16^1_1\,
32^2_\II\, 
64^{-2}_\II\,
128^{1}_1\, 
256^{1}_{-1}\,
512^{-4}_\II
\end{equation}
It is sometimes convenient and sometimes annoying to allow
trivial ($0$-dimensional) terms in a Jordan decomposition.

The main difficulty of $2$-adic lattices is that a given lattice may
have several inequivalent Jordan decompositions.  The purpose of the
Conway-Sloane calculus is to allow one to move easily between all 
possible isometry classes of Jordan decompositions.  Some of the data
in the Jordan symbol remains invariant under these moves.  First, if
one has two Jordan decompositions for the same lattice $L$, then each
term in one has the same dimension as the term of that scale in the
other.  (Scaling reduces the general case to the integral case, which
follows by considering the structure of the abelian group $L^*/L$.)
Second, the type $\I$ or~$\II$ of the term of any given scale is
independent of the Jordan decomposition.  (One can show this directly,
but we won't need it until after
Theorem~\ref{thm-equivalences-of-2-adic-symbols}, which
implies it.)  The signs and oddities of the constituents are not usually
invariants of~$L$.

We define a \defn{compartment} of a Jordan decomposition just as we
did for fine decompositions: a set of type $\I$ constituents, whose scales
form a sequence of consecutive powers of~$2$, and which is maximal with
these properties.  The example above has three compartments:
$2^{-2}_6
4^3_{-3}$, 
$16^1_1$ and
$128^{1}_1\, 
256^{1}_{-1}$.
The \defn{oddity} of a compartment means the oddity of the
direct sum of its Jordan constituents.
By the definition of the symbols, the compartment oddity
is  the sum (mod~$8$) of the subscripts of those constituents.  
{\bf Caution:} {\it the oddity of a compartment depends on the
Jordan decomposition, and is not an isometry invariant of the
underlying lattice}.  See 
Lemma~\ref{lem-sign-walking-for-2-adic-symbols} for how it can change.
Despite this non-invariance, the oddity of a compartment is useful:

\begin{lemma}[Oddity fusion]
\label{lem-oddity-fusion}
Consider a lattice, a Jordan symbol $J$ for it, and the Jordan symbol
$J'$ got
by reassigning all the subscripts in a compartment, in such a way that
that all resulting terms are legal and the compartment's oddity remains
unchanged.  Then $J,J'$ represent isometric lattices.
\end{lemma}

\begin{proof}
By discarding the rest of $J$ we may suppose it is a single
compartment.  The argument is similar to the odd case of
Theorem~\ref{thm-unimodular-lattices}.  We refine $J,J'$ to fine
symbols $F,F'$.  By hypothesis, the terms of $J'$ have the same signs
as those of~$J$.  It follows that for each scale, the product of the
signs of $F$'s terms of that scale is the same as the corresponding
product for~$F'$.  Therefore sign
walking between equal-scale terms lets us suppose that the signs in~$F$ are
the same as in~$F'$.  Recall from the proof of Lemma~\ref{lem-sign-walking-for-fine-symbols}\eqref{item-sign-walking-same-scale} that
this sort of sign walking amounts to the isomorphisms 
$1^{\e1}_t\,1^{\e'1}_{t'}\iso1^{-\e1}_{t+4}\,1^{-\e'1}_{t'+4}$, which
don't change the compartment's oddity.
    By Lemma~\ref{lem-giver-permutation-and-conversion}, 
giver permutation and conversion also leave the compartment
    oddity invariant.

By giver permutation and possibly swapping $F$ with $F'$, we may
suppose that the
non-matching subscripts are $R$'s in $F$ and $G$'s in $F'$. 
By giver conversion we may suppose $k\leq3$ subscripts fail to match, and 
the
assumed equality of oddities shows $k=0$.   Therefore
    the fine symbols are
the same, so the lattices are isometric.
\end{proof}

\section{2-adic symbols}
\label{sec-2-adic-symbols}

\noindent
One can translate sign walking between fine symbols
to the language of Jordan symbols, but it turns
out to be fussier than necessary.  Things become simpler once we incorporate oddity
fusion into the notation as follows.
The \defn{$2$-adic symbol} of a Jordan decomposition means the
Jordan symbol, except that each compartment is enclosed in brackets and
 the enclosed terms are stripped of their subscripts, whose
 sum in $\Z/8$  is attached to the right bracket as a
subscript.    This is called the compartment's oddity.
For our example~\eqref{EqBasicExample} this yields
\begin{equation*}
1^2_\II\, 
\compartment{2^{-2} 4^3}{3}
\compartment{16^1}{1}
32^2_\II\, 
64^{-2}_\II\,
\compartment{128^{1}\, 256^{1}}{0}
512^{-4}_\II
\end{equation*}
If  a compartment consists of a
single term, such as $\compartment{16^1}{1}$, 
then one usually omits the brackets: 
\begin{equation}
  \label{EqBasicExampleFused}
1^2_\II\, 
\compartment{2^{-2} 4^3}{3}
16^1_1\,
32^2_\II\, 
64^{-2}_\II\,
\compartment{128^{1}\, 256^{1}}{0}
512^{-4}_\II
\end{equation}
Lemma~\ref{lem-oddity-fusion} shows that the isometry type of a
lattice with given $2$-adic symbol is well-defined.

If a compartment has total dimension${}\leq2$ then its oddity is
constrained by its overall sign in the same way as for an odd
unimodular lattice of that dimension.  For compartments of
dimension~$1$ this is the same constraint as before.  In $2$
dimensions, $\compartment{1^+2^-}{0}$ and
$\compartment{1^-2^+}{0}$ are illegal (cannot come from any fine
symbol) because each term $1^+_{\cdots}$ or $2^+_{\cdots}$ would have
$\pm1$ as its subscript, while each term $1^-_{\cdots}$ or
$2^-_{\cdots}$ would have $\pm3$ as its subscript.  There is no way to
choose subscripts summing to~$0$.  The same reasoning
shows that $\compartment{1^+2^+}{4}$ and
$\compartment{1^-2^-}{4}$ are also illegal.

\begin{lemma}[Sign walking for $2$-adic symbols]
\label{lem-sign-walking-for-2-adic-symbols}
Consider the $2$-adic symbol of a Jordan decomposition of a
lattice, and two nontrivial terms of it that
satisfy one of the following:
\begin{enumerate}
\item
\label{item-sign-walking-between-a-compartment-and-neighbor}
they have adjacent scales and different types;
\item
\label{item-sign-walking-in-a-compartment}
they have adjacent scales and type~$\I$, and their compartment either has
dimension${}>2$ or compartment oddity~$\pm2$;
\item
\label{item-sign-walking-between-compartments}
they have type~$\I$, their scales
differ by a factor of~$4$, and the term between
them is trivial.
\end{enumerate}
Then the $2$-adic symbol got by negating their signs, and changing by~$4$ the
oddity of each compartment that contains at least one of the terms, 
represents an isometric lattice.
\end{lemma}

\begin{remark}
As in Lemma~\ref{lem-sign-walking-for-fine-symbols}, one
could also call \eqref{item-sign-walking-between-compartments} sign
jumping.   If the intermediate term were nontrivial of type~$\II$,
then one could achieve the same effect by two moves
\eqref{item-sign-walking-between-a-compartment-and-neighbor}.
If the intermediate term had type~$\I$, then one could achieve
    a \emph{similar} effect by two moves 
    \eqref{item-sign-walking-in-a-compartment}.  It would not
    be quite the same, because both moves would affect the same
    compartment. So its oddity would change by~$4$ twice, ie not at all.
\end{remark}

Our example \eqref{EqBasicExampleFused}
\begin{align*} 
	&
1^2_\II\, 
\compartment{2^{-2}4^3}{3}
16^1_1\,
32^2_\II\, 
64^{-2}_\II\,
\compartment{128^{1}\, 256^{1}}{0}
512^{-4}_\II
\\
\hbox{can walk to\quad}
&
\underbracket[.7pt]{1^{-2}_\II\, 
\leftcompartment{2^2}}\rightcompartment{4^3}{-1}
16^1_1\,
32^2_\II\, 
64^{-2}_\II\,
\compartment{128^{1}\, 256^{1}}{0}
512^{-4}_\II
&\hbox{by \eqref{item-sign-walking-between-a-compartment-and-neighbor},}
\\
\hbox{or\quad}
&
1^2_\II\, 
\compartment{\underbracket[.7pt]{2^{2} 4^{-3}_{\phantom{1}}}}{-1}
16^1_1\,
32^2_\II\, 
64^{-2}_\II\,
\compartment{128^{1}\, 256^{1}}{0}
512^{-4}_\II
&\hbox{by \eqref{item-sign-walking-in-a-compartment},}
\\
\hbox{or\quad}
&
1^2_\II\, 
\leftcompartment{2^{-2}}\underbracket[.7pt]{\rightcompartment{4^{-3}}{-1}
16^{-1}_{-3}}\,
32^2_\II\, 
64^{-2}_\II\,
\compartment{128^{1}\, 256^{1}}{0}
512^{-4}_\II
&\hbox{by \eqref{item-sign-walking-between-compartments}.}
\end{align*}
Underbrackets indicate the terms involved in the moves.
One can also sign walk between scales $16$ and~$32$, between
scales  $64$ and~$128$, and between scales $256$ and~$512$.
But no sign walk is possible between scales $128$ and~$256$,
because that compartment has dimension~$2$ and 
oddity${}\mathrel{\not\equiv}\pm2$
mod~$8$.  More conceptually, a sign walk
 would give the illegal compartment $[128^-256^-]_4$.

\begin{proof}
Refine the Jordan decomposition to a fine decomposition $F$,
apply the corresponding sign walk operation \eqref{item-sign-walking-adjacent-scales-I-and-II}--\eqref{item-sign-walking-scales-differ-by-factor-of-4} from Lemma~\ref{lem-sign-walking-for-fine-symbols} to suitable terms of $F$, and
observe the corresponding change in the Jordan symbol.  
In
case \eqref{item-sign-walking-in-a-compartment} some care is required
because Lemma~\ref{lem-sign-walking-for-fine-symbols} requires both
terms of $F$ to be givers or both to be
receivers.  If the compartment has dimension${}>2$ then we may arrange
this by giver permutation (which preserves compartment
    oddity by Lemma~\ref{lem-giver-permutation-and-conversion}).   
    In dimension~$2$ the hypothesis 
$$\hbox{(compartment oddity)}{}\cong\pm2$$
rules out the case that one is a giver and one a receiver, since
givers and receivers have subscripts $1$ and $-1$ mod~$4$ respectively.
\end{proof}

\begin{theorem}[Equivalence of $2$-adic symbols]
\label{thm-equivalences-of-2-adic-symbols}
Suppose given two lattices with Jordan decompositions.  Then the lattices are
isometric if and only if the $2$-adic symbols of these
decompositions are related by a sequence of the sign walk operations in Lemma~\ref{lem-sign-walking-for-2-adic-symbols}.
\end{theorem}

\begin{proof}
The previous lemma shows that sign walks preserve isometry type.
So suppose the lattices are isometric.
Refine the Jordan decompositions to fine decompositions, apply
Theorem~\ref{thm-equivalence-of-fine-symbols} to obtain a chain of
intermediate fine symbols, and consider the corresponding $2$-adic
symbols.  
    Lemma~\ref{lem-giver-permutation-and-conversion} 
    shows that giver permutation and conversion 
    don't change compartment oddities, so they leave $2$-adic symbol
    unchanged. 
    In the proof of Lemma~\ref{lem-oddity-fusion} we explained
why  sign walking between same-scale terms also has no effect.
The effects of the remaining sign walk operations are recorded in
Lemma~\ref{lem-sign-walking-for-2-adic-symbols}.
\end{proof}

A lattice may have more than one $2$-adic symbol, but the
only remaining freedom lies in the positions of the signs:

\begin{theorem}
\label{thm-same-signs-imply-isometric-same-as-being-identical}
Suppose two given lattices have $2$-adic symbols with the same
scales, dimensions, types and signs.  Then the lattices are isometric
if and only if the symbols are equal, which amounts to having the same
compartment oddities.
\end{theorem}

\begin{proof}
If a $2$-adic symbol $S$ of a lattice $L$ admits a sign walk
affecting the signs of the terms of scales $2^i$, $2^j$ then we write
$\D_{i,j}(S)$ for the resulting symbol.  No sign walks  affect the
conditions for $\D_{i,j}$ to act on $S$, since they don't change the
type of any term or the oddity mod~$4$ of any compartment.  So we may
regard $\D_{i,j}$ as acting simultaneously on all $2$-adic symbols for $L$.
By its description in terms of negating signs and adjusting
compartments' oddities, $\D_{i,j}$ may be regarded as an element of
order~$2$ in the group $\{\pm1\}^T\times(\Z/8)^C$ 
where $T$ is the number of terms
present and $C$ is the number of compartments.  

The assertion of the lemma is that if a sequence of sign walks on $S$ restores the
original signs, then it also restores the original oddities.  We
rephrase this in terms of the subgroup $A$ of
$\{\pm1\}^T\times(\Z/8)^C$ generated by the $\D_{i,j}$.  Namely:
projecting $A$ to the $\{\pm1\}^T$ factor has trivial kernel.  This
is easy to see because the $\D_{i,j}$ are ordered so that they are
$\D_{i_1,j_1},\dots,\D_{i_n,j_n}$ with $ i_1<j_1 \leq i_2<j_2 \leq
\cdots \leq i_n<j_n $.  The linear independence of their
projections to $\{\pm1\}^T$ is obvious.
\end{proof}

To get a canonical symbol for a lattice $L$ one starts with any
$2$-adic symbol $S$ and walks all the minus signs as far left as
possible, canceling them when possible.  To express this formally,
we say two scales {\it can interact} if their terms are as in
Lemma~\ref{lem-sign-walking-for-2-adic-symbols}.  
(We noted in the
previous proof that the ability of two scales to interact is independent of
the particular $2$-adic symbol representing $L$.)  
We define a
\defn{signway} as an equivalence class of scales, under the
equivalence relation generated by interaction.  The language suggests
a pathway or highway along which signs can
move (or cancel).  

One can find the signways without 
fussing with 
the conditions in Lemma~\ref{lem-sign-walking-for-2-adic-symbols}.
First cut the $2$-adic symbol into  the
``trains'' of Conway and Sloane, and then (in rare
cases) cut some of the trains into signways.  
To cut the $2$-adic symbol into trains,
cut between each pair of consecutive powers of~$2$ for
which both terms have type~$\II$.  This includes $0$-dimensional
constituents, which always have type~$\II$.  
To get the signways, inspect each train for ``bad'' compartments
$[q^{\pm1}r^{\pm1}]_4$ or
$[q^{\pm1}r^{\mp1}]_0$,
where $q,r$ are consecutive
powers of~$2$.  Cut in the middle of each such compartment.
To get the  three signways in our example, 
$$
1^2_\II\, 
\compartment{2^{-2} 4^3}{3}
16^1_1\,
32^2_\II
\ \ 
\hbox to0pt{%
    \hss\raise18pt
    \hbox{$\begin{matrix}\hbox{split here to get two trains}
        \\
    \downarrow\end{matrix}$}\hss}%
\ \ 
64^{-2}_\II\,
\leftcompartment{128^{1}}
\ 
\hbox to0pt{%
    \hss\lower18pt
    \hbox{$\begin{matrix}
    \uparrow
    \\
    \hbox{and then split the second train here.}
    \end{matrix}$}\hss}%
\ 
\rightcompartment{256^{1}}{0}
512^{-4}_\II
$$
Each signway has a term of smallest scale, and by sign walking we may
suppose that all minus signs are moved to these terms or canceled with each other.
Then we say the symbol is in \defn{canonical form}, which for our example
is
\begin{equation*}
1^{-2}_\II\, 
\compartment{2^{2} 4^3}{-1}
16^1_1\,
32^2_\II\, 
64^{-2}_\II\,
\compartment{128^{1}\, 256^{-1}}{4}
512^{4}_\II
\end{equation*}
Theorem~\ref{thm-same-signs-imply-isometric-same-as-being-identical} implies:

\begin{corollary}[Canonical form]
\label{cor-canonical-form}
Given lattices $L,L'$ and $2$-adic
symbols $S,S'$ for them in canonical form, $L\iso L'$ if and  only if
$S=S'$.
\qed
\end{corollary}

Conway and Sloane's discussion of the canonical form is in terms
of trains.
They asserted that signs can walk up and down the length of
a train,  so that after walking signs leftward, there is at most one
sign per train.  
But this is not true, as pointed out in \cite{Allcock}.
One cannot walk the 
minus sign in 
$\compartment{128^{1}\, 256^{-1}}{4}$
leftward
because there is no way to assign the 
subscripts in $128^-_{\pm3}\,256^+_{\pm1}$ so
that the compartment has oddity~$0$.  

\begin{example}
  \label{EgOrthocomplement}
  As an extended demonstration of sign walking, we determine the
  lattices $M$ with the property that $M\oplus\qf{2,2}\iso L$ where $L$ is 
  from \eqref{EqBasicExampleFused}.
  Note that $\qf{2,2}=2^2_2$.
  Obviously we require
  \begin{equation*}
    M=\underbracket[.7pt]{1^{\pm2}_\II} 
    \underbracket[.7pt]{4^{\pm3}_? 16^{\pm1}_? 32^{\pm2}_\II}
    \underbracket[.7pt]{64^{\pm2}_\II \leftcompartment{128^{\pm1}}}
    \underbracket[.7pt]{\rightcompartment{256^{\pm1}}{?} 512^{\pm4}_\II}
  \end{equation*}
  We have marked the signways with underbrackets.
   The 3rd and 4th of these
  become the  2nd and 3rd signways of~$L$ after summing with $2^2_2$.
  No sign walking is possible between distinct signways.  So the isomorphism $M\oplus2^2_2\iso L$
  shows that the terms in these signways in $M$ can be taken to
    coincide with
  the corresponding terms in  $L$.
  Next, the first two signways of $M$ fuse with the $2^2_2$ summand to form the first signway of $L$.
  The overall sign of this in $L$ is~$-$, so the total number of $-$ signs in the
  first two signways of $M$ must be odd.  By sign walking in the second signway of $M$, we reduce to
  \begin{equation*}
M\iso\Bigl(
  1^{-2}_\II 4^3_t 16^1_u 32^2_\II
  \hbox{\ \ or\ \ }
  1^{2}_\II 4^{-3}_t 16^1_u 32^2_\II
\Bigr)
\oplus
64^{-2}_\II
\compartment{128^1 256^1}{0}
512^{-4}_\II
  \end{equation*}
  where $t$ and $u$ are unknowns.
  Now we sum with $2^2_2$ to get 
  \begin{equation*}
L\iso\Bigl(
  1^{-2}_\II\compartment{2^2 4^3}{2+t} 16^1_u 32^2_\II
  \hbox{\ \ or\ \ }
  1^{2}_\II \compartment{2^2 4^{-3}}{2+t} 16^1_u 32^2_\II
\Bigr)
\oplus
\cdots
  \end{equation*}
  Then we sign walk between the first two terms, or between the second and third, to make the
  signs match those in \eqref{EqBasicExampleFused}.  That is,
  \begin{equation*}
L\iso\Bigl(
  1^{2}_\II\compartment{2^{-2} 4^3}{6+t} 16^1_u 32^2_\II
\Bigr)
\oplus
\cdots
  \end{equation*}
  Both this and \eqref{EqBasicExampleFused} 
  represent $L$, and the signs match, so the subscripts must too.
  Therefore $6+t=3$ and $u=1$.  That is,
  \begin{equation*}
M\iso
1^{\pm2}_\II 4^{\mp3}_5 16^1_1 32^2_\II
64^{-2}_\II
\compartment{128^1 256^1}{0}
512^{-4}_\II
  \end{equation*}
  where one  ambiguous sign is~$+$ and the other is~$-$.
  The two possibilities are distinct because 
    both are in canonical form.
  It follows that the isometry group of $L$ has two orbits on summands isomorphic to $\qf{2,2}$. 
\end{example}

One can use the ideas of the proof of Theorem~\ref{thm-same-signs-imply-isometric-same-as-being-identical} to give numerical
invariants for lattices, if one prefers them to a canonical form.  
The following invariants come from
Theorem~10 of \cite[Ch.\ 15]{SPLAG}, which is proven in \cite{Xu}.  
One records the scales, dimensions and types, the \defn{adjusted oddity} of each compartment, and 
the overall sign of each signway 
(the product of the signs of the signway's terms).  Here
the adjusted oddity of a compartment means its oddity plus~$4$ for
each $-$ sign appearing in its 1st, 3rd, 5th,${}\ldots$ position, with
each $-$ sign after that compartment counted as occurring in the
``$(k+1)$st'' position, where $k$ is the number of terms in the
compartment.  
It is easy to check that 
sign walking leaves these quantities unchanged.

As an example,  the adjusted oddity of the compartment
$\compartment{2^{-2}4^3}{3}$ in
$$
1^2_\II\, 
\compartment{2^{-2}4^3}{3}
16^1_1\,
32^2_\II\, 
64^{-2}_\II\,
\compartment{128^{1}\, 256^{1}}{0}
512^{-4}_\II
$$
is $3+4+4+4=7$.  The~$3$ is the ordinary oddity, and the
first~$4$ is because the compartment has sign~$-$ in its
first position.  The last two $4$'s come from the signs on
the terms of scales $64$ and~$512$.  For purposes of the
adjusted oddity, each of these counts as appearing in
the third position of the compartment, hence contributes
$4$ to the adjusted oddity.

These invariants are clumsy because of the definition
of adjusted oddity, which has the ugly feature that it
depends
on signs outside the signway containing the relevant compartment.  
This goes against the principle
that simplified Example~\ref{EgOrthocomplement}: 
distinct signways are  isolated from each other.

Furthermore, these invariants
are really
just a complicated way of recording the canonical form while
pretending not to.
We will show how to construct the  unique $2$-adic symbol in canonical form having the same
invariants as any chosen $2$-adic lattice.  To do this we first
observe that the types of the constituents determine the
compartments.  The oddities of the compartments are
the same mod~$4$ as the given adjusted oddities.  This data controls
which sign walks are possible, hence determines the signways.
We set the sign of the first term of each signway equal to
the given overall sign of that signway, and the other signs to~$+$.
The signs then allow one to compute the compartment oddities from the
adjusted oddities.  


\section{Equivalences between fine decompositions}
\label{sec-equivalences-between-fine-decompositions}

\noindent
In this section we give the deferred proof of
Theorem~\ref{thm-equivalence-of-fine-symbols}: two fine symbols
represent isometric lattices if and only if they are related by sign
walks and giver permutation and conversion.  
Logically, it belongs anywhere between Theorems \ref{thm-unimodular-lattices} and~\ref{thm-equivalences-of-2-adic-symbols}.
The next two lemmas  are standard; our proofs are adapted from Cassels
\cite[pp.~120--122]{Cassels}.

\begin{lemma}
\label{lem-perps-of-odd-norm-vectors}
Suppose $L$ is an integral lattice, that $x,x'\in L$ have the same odd
norm, and that their orthogonal complements $x^\perp,x'^\perp$ are
either both odd or both even.  Then $x^\perp\iso x'^\perp$.
\end{lemma}

\begin{proof}
First, $(x-x')^2$ is even.  If it is twice an odd number then
the reflection in $x-x'$ is an isometry of~$L$.  This reflection exchanges
$x$ and $x'$, so it gives an isometry between $x^\perp$ and
$x'^\perp$.  This argument applies in particular if $x\cdot x'$ is
even.  So we may restrict to the case that $x\cdot x'$ is odd and
$(x-x')^2$ is divisible by~$4$.  Next, note that $(x+x')^2$ differs
from $(x-x')^2$ by $4x\cdot x'\cong4$ mod~$8$.  So by replacing $x'$
by $-x'$ we may suppose that $(x-x')^2\cong4$ mod~$8$.  This
replacement is harmless because $\pm x'$ have the same orthogonal
complement.  

If it happens that $(x-x')\cdot
L\sset2\Z_2$ then the reflection in $x-x'$ preserves $L$ and we may
argue as before.  So suppose some $y\in L$ has odd inner product with
$x-x'$. 
Then the inner product matrix of $x,x-x',y$ is
$$
\begin{pmatrix}
1&0&?\\
0&0&1\\
?&1&?
\end{pmatrix}
\quad\hbox{mod~$2$,}
$$ which has odd determinant.  Therefore these three vectors span a
unimodular summand of $L$, so $L$ has a Jordan decomposition whose
unimodular part $L_0$ contains both $x$ and $x'$.  Note that $x$'s
orthogonal complement in $L_0$ is even just if its orthogonal
complement in $L$ is, and similarly for $x'$.  So 
by discarding the rest of
the decomposition we may suppose $L=L_0$, without losing our
hypothesis that $x^\perp,x'^\perp$ are both odd or both even.
Now,  $x^\perp$ is unimodular
with $\det(x^\perp)=(\det L)/x^2$ and oddity
$o(x^\perp)=o(L)-x^2$, and similarly for $x'$.  Since
$x^2=x'^2$, Theorem~\ref{thm-unimodular-lattices} implies $x^\perp\iso x'^\perp$.
\end{proof}

\begin{lemma}
\label{lem-perps-of-even-unimodular-sublattices}
Suppose $L$ is an integral lattice and $U,U'\sset L$ are isometric
even unimodular sublattices.  Then $U^\perp\iso U'^\perp$.
\end{lemma}

\begin{proof}
$U\oplus\qf{1}$ has an orthogonal basis  $x_1,\dots,x_n$ by
  Lemma~\ref{lem-odd-unimodular-admits-orthogonal-basis}, and we write $x_1',\dots,x_n'$ for the basis for $U'\oplus\qf{1}$
corresponding to it under some isometry $U\iso U'$.  Apply the
previous lemma $n$ times, starting with $L\oplus\qf{1}$.  (In the $n$th application we
need the observation that the orthogonal complements of $U,U'$ in $L$
are both even or both odd.  This holds because these orthogonal
complements are even or odd
according to whether $L$ is.)
\end{proof}

\begin{lemma}
\label{lem-can-get-norm-1-term-or-none-exists}
Suppose $L$ is an integral lattice and that $1^+_G$ is a term in some
fine symbol for $L$.  
Then we
may apply a sequence of sign
walking and giver permutation and conversion operations to transform
any other fine symbol $F$ for $L$ into one possessing  a
term $1^+_G$.
\end{lemma}

\begin{proof}
We claim first that after some of these operations we may suppose $F$
has a term $1^{+}_{\cdots}$.
Because $L$ is odd, $F$'s terms of scale~$1$ have the form $1^{\pm}_{\RorG}$.
If $F$ has more
than one such term then we can obtain a sign~$+$ by sign walking,
so suppose it has only one term, of sign~$-$.  If there are type~$\I$  terms of
scale~$4$ then again we can use sign walking, so suppose
all scale~$4$ terms have type~$\II$.  We can do the same thing if there are
any terms~$2^{\pm2}_\II$.  Or terms $2^{\pm1}_{\RorG}$, if the
compartment consisting of the scale $1$ and~$2$ terms has
at least two givers or two receivers.  This holds in particular if
there is more than one term of scale~$2$.  So
we have reduced to the case
$$
F=1^-_{\RorG}\,4_\II^{\cdots}\,8_{\cdots}^{\cdots}\cdots
\quad\hbox{or}\quad
F=1^-_{\RorG}\,2^{\pm}_{\RorG}\,4_\II^{\cdots}\,8_{\cdots}^{\cdots}\cdots
$$
where in the latter case one subscript is $G$ and the other is~$R$.  
(Here and below, the superscript and subscript dots indicate
any possibilities for the number of terms 
at that scale, and their decorations in that position. 
In particular, there might be no terms of that scale.
The dots at the end indicate terms of higher scale than the
ones already listed.)
By giver permutation we may
suppose
$$
F=1^-_{\RorG}\,4_\II^{\cdots}\,8_{\cdots}^{\cdots}\cdots
\quad\hbox{or}\quad
F=1^-_G\,2^{\pm}_R\,4_\II^{\cdots}\,8_{\cdots}^{\cdots}\cdots
$$
None of these cases occur, because 
these lattices don't represent $1$ mod~$8$, contrary to the hypothesis
that some fine decomposition has a term~$1^+_G$.
This non-representation
is easy to see because $L$ is $\qf{\pm3}$ or $\qf{5,-2}$ or
$\qf{5,6}$, plus a lattice in which all norms are divisible by~$8$.  

So we may suppose $F$ has a term $1^+_{\cdots}$, and must show that after
further operations we may suppose it has a term $1^+_G$.  
We are done unless our term $1^+_{\cdots}$ is $1^+_R$. If the
compartment $C$ containing it has any givers then we may use giver
permutation to complete the proof.   So suppose $C$ consists of
receivers.  If there are $4$ receivers then we may
convert them to givers, reducing to the previous case.  
If $C$ has two terms of
different scales, neither of which is our $1^+_R$ term, then
we may use sign walking to convert them to givers, again reducing to a known case.   Only a few
cases remain, none of which actually occur, by a similar argument to
the previous paragraph.

Namely, after more sign
walking we may take  $F$ to be
\begin{displaymath}
\bigl(
1^+_R\,2^{\pm}_R\hbox{ or } 1^+_R\,2^+_R\,2^{\pm}_R
\bigr)
4^{\cdots}_\II 
\cdots
\quad\hbox{or}\quad
\bigl(1^+_R \hbox{ or } 1^+_R \, 1^{\pm}_R \hbox{ or }
1^+_R\,1^+_R\,1^{\pm}_R
\bigr)2^{\cdots}_\II 
\cdots
\end{displaymath}
The first set of possibilities is
\begin{align*}
\Bigl(
\qf{-1,-2}
&\hbox{ or }
\qf{-1,6}
\hbox{ or }
\qf{-1,-2,-2}
\hbox{ or }
\qf{-1,-2,6}
\Bigr)\\
&\oplus
\hbox{(a lattice with all norms divisible by~$8$)}
\end{align*}
none of which represent $1$ mod~$8$.  The second set
of possibilities is
\begin{align*}
\Bigl(
\qf{-1}
\hbox{ or }
\qf{-1,-1}
&\hbox{ or }
\qf{-1,3}
\hbox{ or }
\qf{-1,-1,-1}
\hbox{ or }
\qf{-1,-1,3}
\Bigr)\\
&\oplus
\hbox{(a lattice with all norms divisible by~$4$)}
\end{align*}
and only the last two cases represent $1$ mod~$8$.  But in these cases every
vector $x$ of norm $1$ mod~$8$ projects to $\xbar:=(1,1,1)$ in $U/2U$, where $U$
is the summand $\qf{-1,-1,-1}$ or $\qf{-1,-1,3}$.  There are no
odd-norm vectors orthogonal to $x$ since the orthogonal complement of
$\xbar$ in $U/2U$ consists entirely of self-orthogonal vectors.  So while
these lattices admit norm~$1$ summands, they do not admit fine
decompositions with $1^+_G$ terms.  
\end{proof}

\begin{lemma}
\label{lem-can-get-even-unimodular-term-or-none-exists}
Suppose $\e=\pm$.  Then Lemma~\ref{lem-can-get-norm-1-term-or-none-exists} holds with  $1^{\e2}_\II$ in
place of $1^+_G$.
\end{lemma}

\begin{proof}
If $F$ has two terms of scale~$1$, or a scale~$2$ term of
type~$\I$, then we can use sign walking.  The only remaining case is
$F=1^{-\e2}_\II\,2^{\cdots}_\II\,4_{\cdots}^{\cdots}\cdots$.   
Write $U$ for the $1^{-\e2}_\II$ summand and note that any two elements of
    $L$ with the same image in $L/(2U\oplus U^\perp)=U/2U$ 
    have the same norm mod~$4$.  Direct
calculation shows that the norms of the nonzero elements of $U/2U$ are
$0,0,2$ or $2,2,2$ mod~$4$, depending on $\e$.  Now consider the
summand $U'\iso1^{\e2}_\II$ of $L$ that we assumed to exist.  By
considering norms mod~$4$ we see that $U'/2U'\to U/2U$ cannot be injective, so
it must have image $0$ or~$\Z/2$.  
Since all self-inner products in $U/2U$ vanish, we obtain
the absurdity that all inner products in $U'$ are even.
\end{proof}

\begin{proof}[Proof of Theorem~\ref{thm-equivalence-of-fine-symbols}]
The ``if'' part has already been
proven in Lemmas \ref{lem-sign-walking-for-fine-symbols} and~\ref{lem-giver-permutation-and-conversion}, so we prove ``only if''. 
We assume the result for all lattices of lower dimension.  By scaling
by a power of~$2$ we may suppose $L$ is integral and some inner
product is odd, so each of $F$ and $F'$ has a nontrivial unimodular term.

First suppose $L$ is odd, so the unimodular terms of $F$ and $F'$
have type~$\I$.  By rescaling $L$ by an odd number we may suppose $F$ has
a term $1^+_G$.  By Lemma~\ref{lem-can-get-norm-1-term-or-none-exists} we may apply our moves to $F'$ so
that it also has a term $1^+_G$.   The orthogonal complements of the corresponding summands of
$L$ are both even (if the unimodular Jordan constituents are $1$-dimensional) or both odd
(otherwise).  By Lemma~\ref{lem-perps-of-odd-norm-vectors} these orthogonal complements are
isometric.  They  come with fine decompositions, given by 
the remaining terms in $F,F'$.  By induction on dimension these
fine decompositions are equivalent by our moves.

If $L$ is even then the same argument applies, using Lemmas \ref{lem-can-get-even-unimodular-term-or-none-exists}
and~\ref{lem-perps-of-even-unimodular-sublattices} in place of Lemmas \ref{lem-can-get-norm-1-term-or-none-exists} and~\ref{lem-perps-of-odd-norm-vectors}.
\end{proof}

\end{document}